\documentclass[12pt, reqno]{amsart}
\usepackage{amsfonts}
\usepackage{bbm}
\usepackage{amscd,amsfonts}
\usepackage{amssymb, eucal, amsfonts, amsmath, xypic, latexsym, tikz}
\usepackage{pifont}
\usepackage{mathrsfs,color}
\usepackage{amsthm,indentfirst,bm,fancyhdr,dsfont}
\usepackage{graphicx}
\usepackage[all]{xy}
\usepackage[CJKbookmarks=true]{hyperref}

\usepackage{mathrsfs}
\usepackage{amsmath}
\usepackage{amssymb}
\usepackage{hyperref}

\newtheorem{theorem}{Theorem}[section]
\newtheorem{lemma}[theorem]{Lemma}
\newtheorem{prop}[theorem]{Proposition}

\newtheorem{corollary}[theorem]{Corollary}
\theoremstyle{definition}

\newtheorem{rem}[theorem]{Remark}
\newtheorem{remark}[theorem]{Remark}

\newtheorem{jacon}[theorem]{Jacobian Conjecture}

\numberwithin{equation}{section}

\def\mmm{\mathfrak{m}}

\def\calm{\mathcal{M}}

\def\bfb{\mathbf{b}}
\def\bd{\mathbf{d}}
\def\ba{\mathbf{a}}
\def\bc{\mathbf{c}}

\def\bk{\mathbf{k}}

\def\bff{\mathbf{F}}

\def\bbc{\mathbb{C}}
\def\bbf{\mathbb{F}}
\def\bbz{\mathbb{Z}}

\def\bbk{\mathbf{k}}

\def\bba{{\mathbb{A}}}
\def\bbr{{\mathbb{R}}}

\def\bk{\mathbf{k}}

\def\End{\text{End}}
\def\Aut{\textsf{Aut}}

\def\id{\mathsf{id}}

\def\frakr{\frak{R}}
\def\frakg{\frak{G}}
\def\frake{\frak{E}}

{\vskip-\lastskip\medskip
  \noindent
  }%
{\qed\par\medskip
  }

\newcommand{\blue}[1]{{\color{blue}#1}}

\begin{document}

\title[Inverse limits of automorphisms of truncated polynomials]
{Inverse limits of automorphisms of truncated polynomials and applications related to {Jacobian} conjecture}

\author{Hao Chang, Bin Shu, Yu-Feng Yao}

\address{School of Mathematics and Statistics, and Hubei Key Laboratory of Mathematical Sciences, Central China Normal University, Wuhan, 430079, China}\email{chang@ccnu.edu.cn}
\address{School of Mathematical Sciences, Ministry of Education Key Laboratory of Mathematics and Engineering Applications \& Shanghai Key Laboratory of PMMP,  East China Normal University, No. 500 Dongchuan Rd., Shanghai 200241, China} \email{bshu@math.ecnu.edu.cn}

\address{Department of Mathematics, Shanghai Maritime University, Shanghai, 201306, China.}\email{yfyao@shmtu.edu.cn}



\thanks{This work is partially supported by the National Natural Science Foundation of China (Grant Nos. 12071136 and 12271345), supported in part by Science and Technology Commission of Shanghai Municipality (No. 22DZ2229014).}

\begin{abstract} In this note, we investigate Jacobian conjecture through investigation of automorphisms of polynomial rings in characteristic $p$. Making use of the technique of inverse limits, we show that under Jacobian condition for a given homomorphism $\varphi$ of the polynomial ring $\bk[x_1,\ldots,x_n]$, if $\varphi$ preserves the maximal ideals, then $\varphi$ is an automorphism.
\end{abstract}

\maketitle
\setcounter{tocdepth}{1}\tableofcontents

\section{Jacobian Conjecture}
Let $\bbf=\bbc$ or $\bbr$ in this section. Let $\theta: \bba_\bbf^n\rightarrow \bba_\bbf^n$ be a morphism of affine spaces which sends $(t_1,\ldots,t_n)\in \bba_\bbf^n$ to $(\theta_1(t_1,\ldots,t_n),\ldots, \theta_n(t_1,\ldots,t_n))\in \bba_{\bbf}^n$, and $J(\theta)$ denote the Jacobian matrix which is $({\partial\theta_i\over \partial t_j})_{i,j=1,\ldots,n}$. Set $\vartheta:=\theta^*$, which is the comorphism of $\theta$.

\begin{jacon}
{
If $\operatorname{det}(J(\theta)) \in \mathbb{F}^{\times}$, then is an automorphism of the affine algebraic variety $\mathbb A^n$, that is to say,   $\vartheta$ is an automorphism of the polynomial ring $\mathbb{F}\left[x_1, \ldots, x_n\right]$ where $x_i$ means the $i$ th coordinate function for $i=1, \ldots, n$.
}

\end{jacon}

 Also, we can talk about the Jacobian matrix $J(\vartheta)=(\partial_j\vartheta(x_i))_{n\times n}$. Obviously, both $J(\vartheta)$ and $J(\theta)$ have the same meaning. We will not distinguish them.

\begin{rem}
The reader is referred to \cite{B, DP, KM, Su} {\sl{etc.}} for the recent progress of Jacobian conjecture.
\end{rem}

\section{Characteristic $p$}
Suppose $\bk$ is an algebraically closed field of characteristic $p>0$.
Let $A(n)$ denote the truncated polynomial ring which is  by definition the quotient of $\bk[x_1,\ldots,x_n]$ by the ideal $\langle x_1-a_1, \cdots, x_n-a_n\rangle$.
As to $\Aut(A(n))$, any element is determined by its action on the generators $
x_i$, $i = 1, \ldots, n$. There is a criterion judging when an algebra endomorphism $\varphi$ of
$A(n)$ is an automorphism, that is when and only when $\varphi$ stabilizes the unique maximal ideal
of $A(n)$ and the determinant $\textsf{det}({J}(\varphi))$  is invertible in $A(n)$ where ${J}(\varphi)$ denotes the Jacobian matrix $(\partial_j\varphi(x_i))_{n\times n}$ (see \cite{Wil}).

\section{Inverse limits for truncated polynomials in characteristic $p$}
Still suppose $\bbk$ is a field of characteristic $p>0$. Denote by $P(n)$ the polynomial ring $\bbk[x_1,\ldots,x_n]$.

For a given $n$-tuple $\ba=(a_1,\ldots,a_n)\in\bk^n$, set $\mmm_\ba=\langle x_1-a_1, \cdots, x_n-a_n\rangle$, which is a maximal ideal of $P(n)$ generated by $x_i-a_i$ for $i=1,2,\dots,n$. Let $R:=P(n)$ and
$$R_s(\ba)=P(n)\slash \langle(x_1-a_1)^p, \ldots,(x_n-a_n)^p\rangle^{p^{s-1}}$$
 where $s=1,2,\ldots$, and $\langle(x_1-a_1)^p, \ldots,(x_n-a_n)^p\rangle$ denotes the ideal of $P(n)$ generated by $(x_1-a_1)^p, \ldots,(x_n-a_n)^p$.  Especially,  $R_1(\ba)$ is the ``restricted" truncated polynomial ring $A_\ba(n)=P(n)\slash \langle(x_1-a_1)^p, \ldots,(x_n-a_n)^p\rangle$.

Keep in mind that $\bbk[x_1,\ldots,x_n]=\bbk[x_1-a_1,\ldots,x_n-a_n]$. If we take $y_i=x_i-a_i$, then $R_s(\ba)$ can be regarded $R_s:=R_s(\mathbf{0})$ for $\mathbf{0}=(0,\ldots,0)$. So, any results for the ``restricted" truncated polynomial, i.e. the case with $\ba=\mathbf{0}$, can be converted to the version of the general case $\ba$.

\begin{lemma}\label{lem: 3.0}
The automorphism group $\Aut(R_s(\ba))$ consists of ring homomorphisms $\sigma$ satisfying the following items
\begin{itemize}
\item[(\blue{1})] The ideal $\overline{\mmm_\ba}$ generated by $\overline{x_i-a_i}$ with $i=1,\ldots,n$ is preserved by $\sigma$ where $\overline{x_i-a_i}$ denotes the image of $x_i-a_i$ in the quotient.
\item[(\blue{2})] The determinant $\det(J(\sigma))$ is an invertible element {in $R_s(a)$.}
\end{itemize}
\end{lemma}

\begin{proof}
(1) Let $\sigma\in\Aut(R_s(\ba))$.  We show that the following two items hold.

(i) Suppose $\sigma(\overline{x_i-a_i})=\sum\limits_{j=1}^n\overline{(x_j-a_j)q_{ij}}+\overline{b_i}$ for $1\leq i\leq n$, where all $q_{ij}\in P(n), b_i\in\bbk$. Since $\sigma$ is a ring homomorphism of $P(n)$, we have $0=\sigma(\overline{x_i-a_i})^{p^s}=\overline{b_i}^{p^s}$. This implies that $b_i=0$ for all $1\leq i\leq n$, that is, $\sigma(\overline{\mmm_\ba})\subseteq \overline{\mmm_\ba}$, as desired.

(ii) It follows from (i) that we can write   $\sigma(\overline{x_i-a_i})=\sum\limits_{j=1}^nc_{ij}\,\overline{x_j-a_j}+u_i$ for $c_{ij}\in\bbk, u_i\in\overline{\mmm_\ba}^2$. We claim that the $n\times n$ matrix $(c_{ij})$ is invertible. Otherwise, there exist some $k_i\in\bbk$ and not all of them are zero such that $\sigma(\sum\limits_{i=1}^n k_i\overline{x_i-a_i})\in\overline{\mmm_\ba}^2$. Moreover, since $\sigma^{-1}(\overline{\mmm_\ba})\subseteq \overline{\mmm_\ba}$ by similar arguments as in (i), we see that $\sigma^{-1}(\overline{\mmm_\ba}^2)\subseteq \overline{\mmm_\ba}^2$. Then
$$\sum\limits_{i=1}^n k_i\overline{x_i-a_i}=\sigma^{-1}\Big(\sigma\big(\sum\limits_{i=1}^n k_i\overline{x_i-a_i}\big)\Big)\in\overline{\mmm_\ba}^2,$$
which is a contradiction. Hence, the claim follows, i.e., $\det(c_{ij})\in\bbk^{\times}$. Since $$\det(J(\sigma))\equiv\det(c_{ij})\,(\text{mod}\, \overline{\mmm_\ba})$$ and $\overline{\mmm_\ba}$ is nilpotent, it follows that $\det(J(\sigma))$ is invertible.

(2) Let $\sigma$ be a ring homomorphism of $R_s(\ba)$ satisfying (i) and (ii), we aim at showing that  $\sigma\in\Aut(R_s(\ba))$. We can assume that
$\sigma(\overline{x_i-a_i})=\sum\limits_{j=1}^nc_{ij}\,\overline{x_j-a_j}+u_i$ for $c_{ij}\in\bbk, u_i\in\overline{\mmm_\ba}^2$, where the matrix $A=(c_{ij})$ is invertible by (ii).  We can further assume that
\[
u_i=(\overline{x_1-a_1},\cdots, \overline{x_n-a_n})B_i\begin{pmatrix}
\overline{x_1-a_1}\\
\vdots\\
\overline{x_n-a_n}\\
\end{pmatrix}+u_i^{\prime}
\]
where $B_i$ is a $n\times n$ matrix, $u_i^{\prime}\in\overline{\mmm_\ba}^3$ for $1\leq i\leq n$.
Set $A^{-1}=(d_{ij})$. Define a ring homomorphism $\sigma_1$ of $R_s(\ba)$ with \[\sigma_1(\overline{x_i-a_i})=\sum\limits_{j=1}^nd_{ij}\,\overline{x_j-a_j}-
(\overline{x_1-a_1},\cdots, \overline{x_n-a_n})AB_iA^{\prime}\begin{pmatrix}
\overline{x_1-a_1}\\
\vdots\\
\overline{x_n-a_n}\\
\end{pmatrix},\,\,1\leq i\leq n.
\]
Then $\sigma_1\in\Aut(R_s(\ba))$ by (1). Then $\sigma_1\sigma$ is also a ring homomorphism of $R_s(\ba)$ satisfying (i) and (ii) with $\sigma(\overline{x_i-a_i})=\overline{x_i-a_i}+v_i$ with $v_i\in\overline{\mmm_\ba}^3$ for $1\leq i\leq n$. Similar arguments yield that there exist $\sigma_2,\cdots,\sigma_m\in\Aut(R_s(\ba))$ such that
$\sigma_m\cdots\sigma_1\sigma=\id$, so that $\sigma=\sigma_1^{-1}\cdots\sigma_m^{-1}\in\Aut(R_s(\ba))$ as desired.
\end{proof}

\begin{lemma}\label{lem: 3.1} The following statements hold.
 \begin{itemize}
  \item[(1)] The system $\{R_s(\ba)\mid s=1,2,\ldots\}$ is an inverse system of unitary commutative associative $\bbk$-algebras. Consider the inverse limit $$\frakr(\ba):=\lim_{\underset{s}{\longleftarrow}}R_s(\ba).$$
     \item[(2)] As an algebra,  $P(n)$ is isomorphic to the subalgebra of $\frakr(\ba)$ which consists of elements $(\pi_1(a),\pi_2(a),\ldots)$ in $\frakr(\ba)$ with $a$ running through $P(n)$ where $\pi_s$ stands for the canonical homomorphism from $P(n)$ to $R_s(\ba)$.
       \end{itemize}
\end{lemma}

\begin{proof}
(1) For any $s\leq t$, there is a canonical homomorphism $\theta_{st}:R_t(\ba)\rightarrow R_s(\ba)$ such that $\theta_{ss}=\id$ and $\theta_{st}\theta_{tl}=\theta_{sl}$ for any $s\leq t\leq l$. Hence, they form an inverse system.

(2) Consider a map from $P(n)$ to $\frakr(\ba)$ which sends $a\in P(n)$ to $(\pi_1(a),\cdots, \pi_s(a),\cdots)\in\frakr(\ba)$. It is clearly a ring homomorphism. We further show it is injective and surjective.
The injectivity follows from the fact that $$\bigcap\limits_{s\geq 1}\langle(x_1-a_1)^p, \ldots,(x_n-a_n)^p\rangle^{p^{s-1}}=0.$$
While the surjectivity is obvious, indeed for any element $(a_1, a_2, \cdots)$ in the image, the first component $a_1$ in $P(n)$ is obviously its preimage.
\end{proof}

For a given maximal ideal $\mmm_\ba$ of $P(n)$, we denote by { $\Aut_{\text{c}}^{\ba}(P(n))$ (resp. $\End^{\ba}_{\text{c}}(P(n))$)} the continuous automorphism group (resp. endomorphism ring) with respect to the $\mmm_\ba$ which consists of automorphisms (resp. endomorphisms) sending ${\mmm_{\ba}}$ onto $\mmm_\ba$. { Here we are talking about the topology on $P(n)$ defined through  $\mmm_\ba$. }

\begin{lemma}\label{lem: 3.2} The following statements hold.
\begin{itemize}
\item[(1)] For $t>s$,  the canonical homomorphism $\theta_{st}:R_t(\ba)\rightarrow R_s(\ba)$ gives rise to the group morphism
 $\Theta_{st}:\Aut(R_t(\ba))\rightarrow \Aut(R_s(\ba))$. We have an inverse limit of algebraic groups over $\bbk$ (which is actually an ind-group)
$$\frakg:=\lim_{\underset{s}{\longleftarrow}}\Aut(R_s(\ba)).$$
And $\frakg$ is a group with identity $(e_1,e_2,\ldots)$ where $e_i$ denotes the identity of $\Aut(R_i(\ba))$ for $i=1,2,\ldots$.
\item[(2)]
As a group, the continuous automorphism group $\Aut_{\text{c}}^\ba(P(n))$ with respect to the maximal ideal $\mmm_\ba$ of $P(n)$
    is isomorphic to the subgroup of $\frakg$ which consists of elements $(\pi_1(\sigma),\pi_2(\sigma),\ldots)$ in $\frakg$ with $\sigma$ running over $\Aut_{\text{c}}^\ba(P(n))$ where $\pi_s$ stands for the canonical homomorphism from $\Aut_{\text{c}}^\ba(P(n))$ to $\Aut(R_s(\ba))$.
\end{itemize}
\end{lemma}

\begin{proof}  (1) For any $\sigma\in \Aut(R_t(\ba))$
it is well-defined that we assign $\sigma$ to an endomorphism of $\End(R_s(\ba))$ via sending $f$ to $\theta_{st}\circ\sigma\circ \theta_{st}^{-1}(f)$ where $\theta_{st}^{-1}(f)$ is an arbitrarily given inverse image of $\theta_{st}$ at $f$. Furthermore, the endomorphism  $\Theta_{st}(\sigma):=\theta_{st}\circ\sigma\circ \theta_{st}^{-1}$ is an automorphism of $R_s(\ba)$ by Lemma \ref{lem: 3.0}, and $\Theta_{st}$ actually gives rise to a group homomorphism from $\Aut(R_t(\ba))$ to $\Aut(R_s(\ba))$.
It is not hard  shown that $\{(\Aut(R_s(\ba)), \Theta_{st})\mid \forall 1\leq s <t\}$ forms an inverse system.

By definition, the elements of $\frakg$ are of the form $(\sigma_1,\sigma_2,\ldots)$ with $\sigma_s\in \Aut(R_s(\ba))$ satisfying $\theta_{st}(\sigma_t)=\sigma_s$ for all $1\leq s<t$. So $\frakg$ forms a group with product
 $$(\sigma_1,\sigma_2,\ldots)\cdot (\sigma'_1,\sigma'_2,\ldots) =(\sigma_1\circ\sigma'_1,\sigma_2\circ\sigma'_2,\ldots).$$

(2) Recall that by definition,  $\Aut_{\text{c}}^\ba(P(n))$ equals $\{\varphi\in \Aut(P(n))\mid \varphi(\mmm_\ba)=\mmm_\ba\}$.
In the following we  denote it by $G$.  We will accomplish the arguments by steps.

(2.1) For any $\varphi\in G$, by definition $\varphi$ preserves the ideal generated by all $(x_i-a_i)^p$. So we can define $\pi_s(\varphi)\in \End(R_s(\ba))$ sending $f$ to $\overline{\varphi(f)}$ for $f\in P(n)$. It is clear that $\pi_s(\varphi)$ satisfies (i) of Lemma \ref{lem: 3.0}. We claim that $\pi_s(\varphi)$  satisfies (ii) too.

 For any $\varphi\in \Aut(P(n))_c$,  by definition we can write
\begin{align}\label{eq: auto ele}
\varphi(x_i-a_i)=\sum_jc_{ij}(x_j-a_j)+\sum_{j}q_{ij}(x_j-a_j)
\end{align}
where $q_{ij}\in \sum_k P(n)(x_k-a_k)$, and  $c_{ij}\in \bbk$. Taking into account the property of $\varphi$ that it preserves the maximal ideal
 $\mmm_\ba$ in combination with its being invertible, we have $\det(c_{ij})$ must be nonzero. Now in $R_s(\ba)$ the ideal generated by $\overline{x_i-a_i}$ is nilpotent. Hence the nonzero of $\det(c_{ij})$ implies that $J(\pi_s(\varphi))$ is invertible.

 Thus, by Lemma \ref{lem: 3.0} we have $\pi_s(\varphi)$ really belongs to $\Aut(R_s(\ba))$.

(2.2) Now we define a map from $\Aut(P(n))_\text{c}$ to $\frakg$ by sending $\varphi$ to
$(\pi_1(\varphi), \pi_2(\varphi),\ldots, \pi_i(\varphi),\ldots)$. It is readily shown that this map is well-defined. We denote it by $\Xi$ which is clearly  a group homomorphism.

(2.3) We show $\Xi$ is an injective map. Suppose $(\pi_1(\varphi), \pi_2(\varphi),\ldots, \pi_i(\varphi),\ldots)$ is the identity of $\frakg$. Then for any $s$, $\pi_s(\varphi)$ is the identity of $\Aut(R_s(\ba))$. So in the expression \eqref{eq: auto ele}, $c_{ij}=1$ when $i=j$, $0$ otherwise. We claim that all $q_{ij}=0$. If not so, then  the image  $\pi_s(\varphi)(x_i-a_i)$  is not $x_i-a_i$ in $\mmm_\ba\slash \mmm_\ba^t$ for some $t$. This means that $\pi_t(\varphi)$ is not the identity, a contradiction. Hence $\varphi$ is indeed the identity of $\Aut(P(n))$.

Summing up, we accomplish the proof.
\end{proof}

By the same arguments, we have the following facts.

\begin{lemma}\label{lem: analogue}
The following statements hold.
\begin{itemize}
\item[(1)] For $t>s$, the canonical homomorphism $\theta_{st}:R_t(\ba)\rightarrow R_s(\ba)$ gives rise to the algebra morphism
 $\vartheta_{st}:\End(R_t(\ba))\rightarrow \End(R_s(\ba))$. We have an inverse limit of algebras over $\bbk$
$$\frake:=\lim_{\underset{s}{\longleftarrow}}\End(R_s(\ba)).$$
In particular, $\frake$ is canonically an algebra.

\item[(2)] There is an algebra homomorphism $\tilde\pi$ from
the continuous endomorphism ring $\End_{\text{c}}^\ba(P(n))$ with respect to the maximal ideal $\mmm_\ba$ of $P(n)$ to the subalgebra of $\frake$ which consists of elements $(\pi_1(\upsilon),\pi_2(\upsilon),\ldots)$ in $\frake$ with $\upsilon$ running over $\End_{\text{c}}^\ba(P(n))$ where $\pi_s$ stands the canonical homomorphism from $\End_{\text{c}}^\ba(P(n))$ to $\End(R_s(\ba))$. Furthermore, $\tilde\pi$ is an isomorphism.

\item[(3)] Continue with (2). Then $\tilde\pi$ sends all injective maps to injective maps; and sends all surjective maps to  surjective maps, where we say $(\pi_1(\phi),\cdots, \pi_s(\phi), \cdots)\in\frak{C}$ is injective (resp. surjective) if all $\pi_s(\phi)$ $(s=1,2,\ldots)$ are injective (resp. surjective).
\end{itemize}
\end{lemma}

\begin{proof} The first two parts can be proved by the same arguments as in the proof of Lemma \ref{lem: 3.2}. We now show (3). For simplicity,  we will briefly call the statements in (3) ``injective statements" and ``surjective statements," correspondingly.

The surjective statement can be easily known by definition. As to the injective statement, for arbitrarily given injective map $\phi\in \End(P(n))_c$ we need to show that $\tilde\pi(\phi):=(\pi_1(\phi),\cdots, \pi_s(\phi), \cdots)$ is an injective map in $\frak{C}$, which means that all $\pi_s(\phi)$ is injective in $\End(R_s(\ba))$. Note that $\phi$ preserves $\sum_{i} R(x_i-a_i)$. So $\phi(x_i-a_i)=\sum_k c_k(x_k-a_k)$ for some $c_k\in P(n)$. Correspondingly, $\phi(x_i-a_i)^p=\sum_k c_k^p(x_k-a_k)^p$. From this fact, it is easily concluded that all
 $\pi_s(\phi)$ are injective because $\phi$ is already injective.
 \end{proof}

We are now in a position to present the following main result in this section, a description of a sufficient condition for an endomorphism of $P(n)$ to be isomorphic.

\begin{prop}\label{prop: 3.3}
The following statements hold.
\begin{itemize}
\item[(1)] Suppose $\phi$ is an endomorphism of $P(n)$ satisfying that  $\textsf{det}(J(\phi))\in \bbk^\times$, and $\phi$ preserves $\mmm_\ba$, then $\phi$ is an automorphism of $P(n)$.

\item[(2)] If $\textsf{det}(J(\phi))\in \bbk^\times$, and $\phi$ sends $\mmm_\ba$ to $\mmm_\bc$ for some $\bc=(c_1,\ldots,c_n)\in \bbk^n$, then $\phi$ is an automorphism of $P(n)$.

\end{itemize}
\end{prop}

\begin{proof} (1) Note that $\phi$ gives rise to an endomorphism of $R_s(\ba)$ for any $s=1,\ldots$. By Lemma \ref{lem: analogue}, we have an injective algebra homomorphism from $\End_{\text{c}}^\ba(P(n))$ to $\frake$. Still  denote the image of $\tilde \pi$ at $\phi$ by $(\cdots,\pi_s(\phi),\cdots)$ by abuse of the notations. Then by the assumption $\det(J(\pi_s(\phi)))$ is invertible, and $\pi_s(\phi)$ preserves the ideal generated by $\overline{x_i-a_i}$ ($i=1,\ldots,n$). By Lemma \ref{lem: 3.0}, $\pi_s(\phi)$ belongs to $\Aut(R_s(\ba))$ satisfying that $\Theta_{st}(\pi_t(\phi))=\pi_s(\phi)$ for all $t>s$. Hence $\tilde \pi(\phi)=(\pi_1(\phi),\pi_2(\phi),\ldots)$ belongs to $\frakg$.

We claim that  $\phi\in \Aut_{\text{c}}^\ba(P(n))$.
Note that all $\pi_s(\phi)$ ($s=1,2,\ldots$) are already known  in $\Aut(R_s(\ba))$, and $\phi$ preserves $\mmm_\ba$, i.e. $\phi\in\End_{\text{c}}^\ba(P(n))$. So by Lemma \ref{lem: 3.2} we have   $\tilde\pi(\phi)\in \tilde\pi(\End(P(n))_c\cap \frakg$. Note that  by Lemma \ref{lem: analogue}  $\tilde\pi$ is an isomorphism which sends surjective maps to surjective maps, and sends injective maps to injective maps. This implies that $\phi$ must fall in $\Aut_{\text{c}}^\ba(P(n))$.


(2) The morphism arising from the mapping $(x_i-a_i)\mapsto (x_i-c_i)$ for $i=1,2,\cdots,n$ obviously gives rise to an automorphism of $P(n)$.  From the combination of this point with (1), the part (2) follows.
\end{proof}

\begin{remark} In \cite{YSL}, the authors studied inverse limits for restricted Lie algebras and their representations.
\end{remark}

\section{A strong Jacobian result in characteristic $p$}
Suppose $\bk$ is an algebraically closed  of characteristic $p>0$. Denote by $\calm$ the set of maximal ideals of $P(n)$. Then $\calm$ coincides with the set $\{\mmm_\ba\mid \ba\in\bbk^n\}$. So we have the following immediate consequence of Proposition \ref{prop: 3.3} which is regarded as an analogue of Jacobian conjecture.

\begin{theorem}\label{thm: Jac p}  Suppose $\bk$ is an algebraically closed field of characteristic $p>0$, and $\varphi$ is a ring endomorphism of $P(n)$ over $\bk$.
If $\textsf{det}(J(\varphi))\in \bbk^\times$, and $\varphi$ preserves $\calm$,  then $\varphi$ is an automorphism of $P(n)$.
\end{theorem}

\begin{proof} It directly follows from Proposition \ref{prop: 3.3}(2).
\end{proof}

\begin{remark} There have been some modular versions of Jacobian conjecture (for example, \cite{MR}).  We believe that the above result could be helpful to  understand more in this topic.
\end{remark}

\section{A counterpart of Theorem \ref{thm: Jac p} in characteristic $0$ }

Suppose $\bbf$ is an algebraically closed filed of characteristic $0$ in this section. Keep the notation $P(n)= \bbf[x_1,\ldots,x_n]$ the polynomial algebra over  $\bbf$. Let
$\phi$ be an algebra endomorphism of $P(n)$. By definition, the Jacobian matrix $J(\phi)$ of $\phi$ is the $n\times n$ size matrix $({{\partial(\phi(x_i))}\over {\partial x_j}})_{i,j=1,\ldots,n}$. Denote by $\mathcal{M}$ the set of all maximal ideals of $P(n)$.

\begin{prop} Suppose $\bbf$ is any given algebraically closed field of characteristic $0$, and $\phi$ is an algebra endomorphism of $P(n)$. If  $\textsf{det}(J(\phi))\in \bbf^\times$, then the following statements are equivalent: (i) $\phi$ preserves  $\calm$. (ii) $\phi$ is surjective.
\end{prop}

\begin{proof} Recall that the set of ring homomorphisms of $P(n)$ to $P(n)$ is  in a one-to-one correspondence with the set of morphisms of affine spaces $\bbf^n$ to $\bbf^n$. Put $X=\bbf^n$. Then there exists a morphism of affine spaces $$\Phi: X\rightarrow X$$ such that $\phi$ coincides with the comorphism $\Phi^*$ of $\Phi$, which sends $(t_1,\ldots,t_n)\in \bbf^n$ to $$(\Phi_1(t_1,\ldots,t_n),\ldots,\Phi_n(t_1,\ldots,t_n))$$
 where each $\Phi_i$ belongs to $P(n)$. Denote by $Y$ the Zariski closure of $\Phi(X)$.  Then $Y$ is an irreducible  closed subset of $X$ defined by the $\ker(\phi)$ which is a prime ideal of $P(n)$.

By definition, the Jacobian matrix  $\text{Jac}(\Phi_1,\ldots,\Phi_n):=({{\partial(\Phi_i)}\over {\partial x_j}})_{i,j=1,\ldots, n}$ coincides with $J(\phi)=({{\partial(\phi(x_i))}\over {\partial x_j}})_{i,j=1,\ldots,n}$ if identifying the coordinates with the coordinate functions. So by assumption, $\det(\text{Jac}(\Phi_1,\ldots,\Phi_n))$ is a nonzero constant in $\bbf$. Hence $\text{Jac}(\Phi_1,\ldots,\Phi_n)$ is of full rank at any point.

Keep in mind the assumption that $\bbf$ is an algebraically closed filed of characteristic $0$.
By \cite[Corollary III.10.7]{Har}, there is a nonempty open subset $V\in Y$ such that $\Phi: \Phi^{-1}(V)\rightarrow V$ is smooth. Suppose the relative dimension of this smooth morphism is $r$ which is the difference $\dim X-\dim Y$.

Now take a point $\bd=(d_1,\ldots,d_n)\in V$. Consider the affine algebraic variety $U$ annihilated by the ideal generated by $\bar\Phi_1:=\Phi_1-d_1,\ldots, \bar\Phi_n:=\Phi_n-d_n$. Then the corresponding Jacobian matrix coincides with $\text{Jac}(\Phi_1,\ldots,\Phi_n)$, consequently  is of full rank. Hence $U$ has dimension zero. Note that $U$ is the fiber of $\Phi$ at $\bd\in V$. So $\dim \Phi^{-1}(\bd)$ is zero. Thanks to the property that the fibers of a smooth morphism  are geometrically regular of equi-dimension $r$ (see \cite[Theorem III.10.2]{Har}), we have $r=\dim X-\dim Y=0$. Note that $Y$ is an irreducible closed subvariety of the irreducible variety $X$. Hence $Y=X$. This means that $\Phi$ is a dominant morphism, equivalently, $\phi$ is an injective endomorphism.

(ii)$\Rightarrow$ (i): By the above arguments, $\phi$ is already injective. So {when Condition (ii) is satisfied},
$\phi$ becomes an automorphism of $P(n)$. Consequently, (i) automatically satisfies.

(i) $\Rightarrow$ (ii): Under the assumption of (i), we want to show that $\phi$ is surjective. It amounts to showing that $\Phi$ is injective.
Keep in mind that $\bbf$ is an algebraically closed field. So any maximal ideal of $P(n)$ is of form $\mmm_\ba:=\sum_{i=1}^nP(n)(x-a_i)$ for some $\ba:=(a_1,\ldots,a_n)\in \bbf^n$.
For a given maximal ideal $\mmm_\ba$, suppose $\phi$ sends  $\mmm_\ba$ to some maximal ideal $\mmm_\bfb$ of $P(n)$ with $\bfb:=(b_1,\ldots,b_n)\in \bbf^n$. Note that $\phi$ is the comorphism of $\Phi$. Then $\Phi$ certainly sends $\bfb\in \bbf^n$ onto $\ba\in \bbf^n$. Actually, if $\ba'=\Phi(\bfb)\in \bbf^n$, from the definition of comorphisms along with the preservation of $\mathcal{M}$ under $\phi$, it follows that the annihilating ideal $\mmm_{\ba'}$ of $\ba'$ must be sent to that of $\bfb$ by $\phi$. The injective property of $\phi$ ensures that $\mmm_{\ba'}=\mmm_{\ba}$. So $\ba'=\ba$.  Conversely, suppose $\Phi$ sends $\bfb$ onto $\ba$. By the definition of comorphisms again, $\phi=\Phi^*$ sends $\mmm_\ba$ into $\mmm_\bfb$.  Furthermore, the assumption of (i) ensures $\phi(\mmm_\ba)=\mmm_\bfb$. Note that $\ba=\ba'$ implies that $\mmm_\bfb=\phi(\mmm_\ba)=\phi(\mmm_\ba')=\mmm_{\bfb'}$. Consequently, the above arguments ensure that $\Phi$ is injective.

The proof is completed.
\end{proof}

\begin{corollary} Suppose $\bbf$ is any given algebraically closed field of characteristic $0$, and $\phi$ is a ring endomorphism of $P(n)$ over $\bbf$. Suppose  $\textsf{det}(J(\phi))\in \bbf^\times$. If additionally  $\phi$ sends any maximal ideals of $P(n)$ onto maximal ideals, then  $\phi$  is an automorphism.
\end{corollary}

\vskip10pt
\subsection*{\textsc{Acknowledgement}} The authors are grateful to the referee for helpful comments and suggestions.

\end{document}